\documentclass[12 pt]{amsart}
\usepackage{amssymb,pdfsync}
\usepackage{graphicx}
\usepackage[mathscr]{eucal}
\usepackage{dsfont}
\usepackage[usenames]{color}
\definecolor{Purple}{rgb}{.7,0.08,0.6} 
 
\usepackage{enumerate}

\theoremstyle{plain}
\newtheorem{Thm}{Theorem}
\newtheorem{Cor}[Thm]{Corollary}
\newtheorem{Prop}[Thm]{Proposition}
\newtheorem{Lem}[Thm]{Lemma}

\theoremstyle{definition}

\newcommand{\B}{\mathbf B}

\renewcommand{\d}{{\tt{d}}}

\newcommand{\g}{g}

\renewcommand{\ge}{g_\epsilon}

\renewcommand{\r}{\rho}

\renewcommand{\Re}{\operatorname{Re}}
\renewcommand{\bar}{\overline}
\renewcommand{\tilde}{\widetilde}
\newcommand{\C}{\mathbb C}

\newcommand{\Cin}{\mathbf C}

\newcommand{\CinII}{\mathbf C^2}

\newcommand{\Cint}{\Cin^{(t)}}
\newcommand{\CintI}{\Cin^{(t), 1}}
\newcommand{\CintII}{\Cin^{(t), 2}}
\newcommand{\Cine}{\mathbf{C}_\epsilon}
\newcommand{\Cker}{C}
\newcommand{\CkerI}{C^1}
\newcommand{\CkerII}{C^2}

\newcommand{\Ctr}{\mathcal C}

\newcommand{\Ctre}{\mathcal C}

\newcommand{\Dt}{D_t}

\newcommand{\Ha}{\mathcal H_\alpha (\bd D)}
\newcommand{\Hp}{\mathbf{H}^p}
\newcommand{\hp}{\mathcal{H}^p}
\newcommand{\htwo}{\mathcal{H}^2(\bd D)}

\newcommand{\Ht}{\mathcal{H}_\vartheta (\bd D)}
\newcommand{\rt}{\rho_t}
\renewcommand{\iota}{s}
\renewcommand{\O}{\Omega}
\newcommand{\Pz}{\mathbf{P}_{\!z}}

\renewcommand{\S}{\mathcal S}
\newcommand{\R}{\mathbb R}

\newcommand{\Cn}{\mathbb C^n}
\newcommand{\Pt}{\Phi_t}

\newcommand{\N}{\mathbb N}

\newcommand{\dee}{\partial}

\newcommand{\deebar}{\overline\partial}

\renewcommand{\N}{\mathcal N}
\newcommand{\bndry}{b}

\numberwithin{equation}{section}

\newcommand{\bd}{\bndry}

\begin{document}
\title
[Hardy Spaces]{Hardy spaces of holomorphic functions for domains in $\Cn$ with minimal smoothness}
\author[Lanzani and Stein]{Loredana Lanzani$^*$
and Elias M. Stein$^{**}$}
\thanks{$^*$ Supported in part by the National Science Foundation, awards DMS-1001304 and DMS-1503612.}
\thanks{$^{**}$ Supported in part by the National Science Foundation, awards
DMS-0901040 and and DMS-1265524.}
\address{
Dept. of Mathematics,       
Syracuse University 
Syracuse, NY 13244-1150 USA}
  \email{llanzani@syr.edu}
\address{
Dept. of Mathematics\\Princeton University 
\\Princeton, NJ   08544-100 USA }
\email{stein@math.princeton.edu}
\begin{abstract} 
We prove various representations and density results for Hardy spaces of holomorphic functions  for two classes of bounded domains in $\Cn$, whose boundaries satisfy minimal regularity conditions (namely the classes $C^2$ and $C^{1,1}$ respectively) together with naturally occurring notions of convexity. 
\end{abstract}
\maketitle
\centerline{\em Dedicated to the memory of Cora Sadosky}
\section{Introduction}
Here we discuss
 the interplay of the holomorphic Hardy spaces with 
  the Cauchy integrals, and with the Cauchy-Szeg\H o projection\footnote{also known as the {\em Szeg\H o projection}.}, for a broad class 
 of bounded domains $D\subset \Cn$ when $n\geq 2$. We make minimal assumptions on the domains' boundary regularity.   While
  such interplay is
   fairly well understood in the context of one complex variable (that is for $D\subset \C$, see \cite{Du}, \cite{Ke-1}, 
   \cite{L}, \cite{LS-1}  and references therein) the situation in higher dimensions presents
    further obstacles, both conceptual and technical in nature, which
  require a different analysis.

 \smallskip
 
 At the root of all these questions
   are three basic issues.
  \smallskip
  
  First, the fact that for any bounded domain $D\subset \Cn$ whose boundary is of class $C^2$ there are two characterizations of the Hardy space $\Hp (D)$, see \cite{S}:
  \smallskip
  
  The holomorphic functions $F$ for which
  \begin{equation}
  \label{E:BVP-hol}
  \|\,\N(F)\|_{L^p(\bndry D, d\sigma)} <\infty\, ,
  \end{equation}
 where  $\N(F)$ is the so-called ``non-tangential maximal function'' of $F$, see \eqref{E:5.2p} below, and $d\sigma$ is the induced Lebesgue measure on the boundary of $D$. Alternatively,
 \begin{equation}\label{E:5.1}
\sup\limits_{0<t<c}\,\int\limits_{w\in\bndry\Dt}\!\!\!|F(w)|^p\,d\sigma_t(w)<\infty.
\end{equation}
where $\{\Dt\}_t$ is an exhaustion of $D$ by appropriate sub-domains $\Dt$. The quantities \eqref{E:BVP-hol} and
\eqref{E:5.1} give equivalent norms of the space $\Hp (D)$, which is also known as the {\em Smirnov class}, see \cite{Du}. Such $F$ have (non-tangential) boundary limits 
$$
f=\dot F\, .
$$
If $\hp (\bndry D, d\sigma)$ denotes the space of such functions, then $\hp (\bndry D, d\sigma)$ is a closed 
subspace of $L^p (\bndry D, d\sigma)$. Moreover the Cauchy-Szeg\H o projection $\S$ can then be defined as the
orthogonal projection of $L^2 (\bndry D, d\sigma)$ onto $\mathcal H^2(\bndry D, d\sigma)$.  
\smallskip

Second, if in addition the domain $D$ is strongly pseudo-convex, we know (by \cite{LS-5}) that $D$ supports an appropriate Cauchy integral $\Cin$ and a corresponding Cauchy transform $\Ctr$, which is a bounded operator on $L^p (\bndry D, d\sigma)$,
 $1<p<\infty$. 
 \smallskip
 
 Third, a combination of the above facts leads to the following consequences:
 
 \begin{itemize}
 \item The approximation theorem (Theorem \ref{T:3-Hardy}): the class of functions holomorphic in a neighborhood of $\bar D$ is dense
 in $\Hp (D)$, $1<p<\infty$, and correspondingly their restrictions to $\bndry D$ are dense in $\hp (\bndry D, d\sigma)$.
 \item[]
 \item The fact that when $f\in L^p(\bndry D, d\sigma)$ then $f\in \hp (\bndry D, d\sigma)$ if, and only if $f=\Ctr (f)$ (Corollary \ref{C:1-Hardy}). Related to this is the conclusion that the image of $L^p(\bndry D, d\sigma)$ under $\Ctr$ is exactly $\hp (\bndry D, d\sigma)$
  (Proposition \ref{P:1-Hardy}).
 \item[]
 \item The identities $\S\Ctr=\Ctr$ and $\Ctr\S =\S$ which hold in $L^2(\bndry D, d\sigma)$ (Proposition \ref{P:6.2.1}). These are crucial in the proof
 (given in \cite{LS-5}) of the $L^p(\bndry D, d\sigma)$-boundedness of $\S$, $1<p<\infty$.
 \item[]
 \item A further characterization of $\hp (\bndry D, d\sigma)$ as the space of those $f\in L^p(\bndry D)$ for which $f=\S (f)$
  (Proposition \ref{L:Har-Sz}).
 \end{itemize}
 \smallskip
 
 The results given in Proposition \ref{P:1-Hardy} and Corollary \ref{C:1-Hardy}, Theorem \ref{T:3-Hardy}, Propositions \ref{P:2-Hardy} and
 \ref{P:6.2.1} also hold for $C^{1,1}$ domains that are strongly $\C$-linearly convex. For these one uses the $L^p$-estimates of the Cauchy-Leray integral given in \cite{LS-4}, together with a suitable
 modification of Lemma \ref{L:aux}. These will appear in a future publication.\\

  We point out that analogous statements apply to the weighted spaces 
  $L^p(\bndry D, \omega\,d\sigma)$ (resp. $\hp (\bndry D, \omega\,d\sigma)$) whenever $\omega$ is a continuous strictly positive density on $\bndry D$:
 these weighted measures include the so-called {\em Leray-Levi measure} considered in \cite{LS-4} and \cite{LS-5}. But note that the space $L^p(\bndry D, \omega\,d\sigma)$ (resp. $\hp (\bndry D, \omega\,d\sigma)$) contains the same elements as $L^p(\bndry D, d\sigma)$ (resp. $\hp (\bndry D, \omega\,d\sigma)$) and the two norms are equivalent, and so we will continue to denote both of these spaces simply by $L^p (\bndry D)$ (resp. $\hp (\bndry D)$). However the distinction between $L^2(\bndry D, d\sigma)$ and $L^2(\bndry D, \omega\,d\sigma)$ 
 becomes relevant when defining the Cauchy-Szeg\H o projections because these spaces have different inner products that give different notions of orthogonality, and so we will distinguish between the two orthogonal projections $\S$ and $\S_\omega$. \\
 
 The structure of this paper is as follows. In Section \ref{S:-backgd} we collect the main definitions along with a few, well-known features of the spaces $\hp (\bndry D)$ for any bounded domain $D$ of class $C^2$. In Section \ref{S: Hardy-Cauchy} we restrict the focus to the strongly pseudo-convex domains and establish various  connections of 
 $\hp(\bndry D)$ with the holomorphic Cauchy integrals that were studied in \cite{LS-5}; in particular we give the proof of Proposition \ref{P:1-Hardy} and of Corollary \ref{C:1-Hardy}, and use these results to establish an operator identity that directly links our Cauchy integrals to the Cauchy-Szeg\H o projection (Proposition  \ref{P:6.2.1}).
 The approximation theorem for $\hp (\bndry D)$ is also proved 
 in Section \ref{S: Hardy-Cauchy} (Theorem \ref{T:3-Hardy}).
 In Section \ref{S:further} we give a further characterization of  $\hp(\bndry D)$ as the range of the Cauchy-Szeg\H o projection for $D$, again for $D$ strongly pseudo-convex and of class $C^2$ (Proposition \ref{L:Har-Sz}). Lastly, in the Appendix we go over the more technical tools and results that are needed to prove Theorem \ref{T:3-Hardy}.\\

 We remark that this paper complements and at the same time is complemented by the paper \cite{LS-5}, in the following sense.
 Part I of \cite{LS-5} is needed to prove the results in Section \ref{S: Hardy-Cauchy} of the present paper. On the other hand, Section \ref{S: Hardy-Cauchy} in this paper (in particular Proposition  \ref{P:6.2.1}) 
 is needed in
 Part II of \cite{LS-5}.
 There is no circularity in our proofs, as Part I of \cite{LS-5} is independent from Part II in that same paper.
    Finally, the results of Part II in \cite{LS-5} (along with Section \ref{S: Hardy-Cauchy} in this paper) lead to the proof of Proposition \ref{L:Har-Sz} in Section \ref{S:further} of the present paper.

\section{Preliminaries}\label{S:-backgd}

Here we recall the main definitions and basic features of the theory of the holomorphic Hardy spaces for a domain $D\subset\Cn$. 
While all results are stated for the induced Lebesgue measure 
$d\sigma$, they are in fact valid for any measure that is equivalent to $d\sigma$ in the sense discussed in the previous section (in particular for the aforementioned  Leray-Levi measure). Thus we will drop explicit reference to the measure and again write $L^p(\bndry D)$, $\hp(\bndry D)$ and so forth.

Most proofs are deferred to
 \cite[Sections 1-5 and 9-10]{S} and \cite{Du}, where references to the earlier literature can also be found. To simplify the exposition we limit this review to the case
  $1<p<\infty$, 
  however we point out that what is needed  for $p=\infty$ is a trivial consequence of $p<\infty$.
 
 The only assumptions on $D$ that we need to make at this stage are that $D$ is bounded and it is of class $C^2$. In fact the class $C^{1+\epsilon}$ would suffice here; the requirement that the domain $D$ is of class $C^2$ will be needed in Section \ref{S: Hardy-Cauchy} and onwards, when we deal with pseudoconvex domains.
 
 \subsection{Small perturbations of the domain}\label{SS:approx-dom}  Starting with our original domain $D$ we will need to construct a family of domains $\{D_t\}$ where $t$ is a small {\em real} parameter.
We recall that since $D$ is 
of class $C^2$ it admits a $C^2$-smooth defining function $\rho: \Cn\to\R$ such that
$$
D=\{\rho<0\},\quad \bndry D=\{\rho=0\}\quad \mbox{and}\quad \nabla\rho\neq 0\ \mbox{on}\ \bndry D.
$$

Then for each real $t$ which is small we consider 
$$
\rt = \rho +t\quad \mbox{and}\quad \Dt =\{\rt<0\},\quad \mbox{so that}\quad \bndry\Dt =\{\rt=0\}.
$$
Note that with this notation we have $D_0=D$ and furthermore $\bar{\Dt}\subset D_0$ if $t>0$, whereas $\bar{D}_0\subset D_t$ if $t<0$. 
We will also need an appropriate bijection between $\bndry D$ and $\bndry \Dt$.
 This 
is given by the exponential map of a normal vector field, which we denote $\Pt: \bndry D\to \bndry \Dt$, and whose
properties are detailed in the Appendix.

\subsection{Hardy spaces}\label{Hardy-def} When $1\leq p< \infty$ and $c$ is a small positive constant, the
 class $\Hp (D)$ is defined as the space of functions $F$ that are holomorphic in $D$ and for which \eqref{E:5.1}
 holds with $D_t$ as above and $t>0$.
(Note that 
therefore $\bar D_t\subset D$.)
We take the norm $\|F\|_{\Hp (D)}$ to be the $p$-th root of the left-hand side of the expression \eqref{E:5.1}.

We recall some basics properties of $\Hp (D)$. First, the space is independent of the specific choice of defining
function that is being used (i.e., another defining function will give an equivalent norm). Also,
one has the following fact about convergence in $\Hp (D)$:
\medskip

{\em If a sequence $\{F_k\}\subset\Hp (D)$ is uniformly bounded in the norm, and $ F_k\to F$ uniformly on compact subsets of $D$, then $F\in \Hp (D)$.}
\medskip

A key feature of
functions in $\Hp$ involves non-tangential convergence. To describe this we fix a convenient non-tangential approach region for each point $w\in\bndry D$. With $\beta>1$ fixed throughout, set 
$\Gamma (w) =\{z\in D\ :\ |z-w|<\beta\, \mbox{dist}(z, \bndry D)\}$, where dist$(z, \bndry D)$ is the usual Euclidean distance of $z\in D$ from the boundary of $D$. Then if $F\in \Hp (D)$ one knows (see \cite[Theorem 10, Section 10,  and its corollary]{S}):
\begin{equation}\label{E:5.2}
 \lim\limits_
{\stackrel{z\to w}{ z\in \Gamma (w)}}
F(z) = \dot F (w)\quad\mbox{exists for a.e.}\ w\in\bd D\, .
\end{equation}
\begin{equation}\label{E:5.2p}
 \N(F) (w) :=  \sup_
{ z\in \Gamma (w)}
|F(z)| \in L^p(\bd D)
\end{equation}
\begin{equation}\label{E:5.2pp}
\|F\|_{\Hp(D)}\ \approx\  \|\dot F\|_{L^p(\bd D)}
\  \approx\  \|\N(F)\|_{L^p(\bd D)}
\end{equation}
The above allow us to define
the Hardy space in terms of the boundary values $\dot F$ of $F\in \Hp (D)$. More precisely, the {\em Hardy space} $\hp (\bd D)$ consists of those function $f\in L^p(\bndry D, d\sigma)$ which arise as 
$$
\dot F=f\, ,\quad \mbox{for some}\quad F\in \Hp(D).
$$
If we take 
$$
\|f\|_{\hp (\bd D)} := \|f\|_{L^p (\bd D)}
$$ 
then \eqref{E:5.2pp} shows that
\begin{equation}\label{E:5.3}
\|f\|_{\hp (\bd D)}\ \approx \ \|F\|_{\Hp(D)}\, 
\end{equation}
and in this sense the spaces $\Hp (D)$ and $\hp (\bndry D)$ are identical with equivalent norms, and we will henceforth refer to either one of $\Hp$ and $\hp$ as ``the Hardy space''. 
 It follows that \eqref{E:BVP-hol} gives an alternate characterization of $\Hp$.

Connected with this is the fact that whenever $F\in \Hp(D)$, one has
\begin{equation*}
F(z) = \Pz (f)\, ,\quad \mbox{with}\quad z\in D,
\end{equation*}
where $\Pz (f)$ is the Poisson integral of $f$, that is
$$
F(z) = \Pz (f)(z)= \!\int\limits_{w\in \bd D} \mathcal P_{\!z} (w)\, f(w)\, d\sigma (w)\, .
$$
The characteristic property of the Poisson kernel $\mathcal P_{\!z} (w)$ is that it gives the solution of the Dirichlet problem for the Laplace operator for $D$ with data $f$. Namely, if $f$ is any continuous function on $\bd D$, and we set
$u(z) = \Pz (f)(z)$, then $u$ is harmonic in $D$, it extends continuously to $\bar D$ and $u(w) = f(w)$
 whenever $w\in \bd D$. More generally, if $f\in L^p(\bndry D)$, then $\dot u (w) = f(w)$ for
  $\sigma$-a.e. 
 $w\in \bd D$ and furthermore, $\|\N (u)\|_{L^p(\bd D)}\lesssim  \|f\|_{L^p(\bd D)}$. The fact that 
 \begin{equation}\label{E:Poisson-bound}
 \sup_{\stackrel{ w\in\bd D}{z\in \Dt}} \mathcal P_{\! z}(w)\leq c_t,\quad \mbox{if}\quad t> 0,
 \end{equation}
 implies that a Cauchy sequence $\{F_k\}_{k}\subset \Hp (D)$ has a subsequence that is uniformly convergent 
on any compact subset of $D$ to a function $F$, which is in fact in $\Hp (D)$, from which the completeness of $\Hp (D)$ is evident. Thus \eqref{E:5.3} shows that $\hp (\bd D)$ is a closed subspace of $L^p(\bd D)$.
  Moreover, one has the following simple characterization.
   \begin{Prop}\label{P:5.4} Suppose $f\in L^p(\bd D)$. Then $f\in \hp (\bd D)$ if, and only if, 
 $u(z)= \Pz (f)$ is holomorphic in $D$.
 \end{Prop}
 Indeed, if $f\in \hp (\bd D)$ then there is $F\in \Hp (D)$ such that $\dot F=f$. But one also has
 $\dot \Pz(f) =f$, and we conclude that $\Pz(f) = F(z), z\in D$, by the uniqueness of the solution of the Dirichlet problem for harmonic functions with data  $f\in L^p(\bndry D)$.
  Conversely, if $\Pz(f)$ is holomorphic in $D$, then by the aforementioned estimates for the solution of the Dirichlet problem with data $f\in L^p(\bndry D)$
    and by the equivalence 
  \eqref{E:5.2pp},
  we have that $\Pz(f)\in \Hp (D)$ and furthermore, that $\dot \Pz (f)= f$, 
  showing that $f\in \hp (\bndry D)$ with $F(z):=\Pz (f)$, $z\in D$.
 \medskip
 
 Proposition \ref{P:5.4} has the following immediate consequence which of course, is interesting only when $p_2>p_1$.
  \begin{Cor}\label{C:5.5} If $f\in \mathcal H^{p_1}(\bndry D)\cap L^{p_2}(\bndry D)$ then $f\in \mathcal H^{p_2}(\bndry D)$.
 \end{Cor}
 Indeed, first note that $u(z)= \Pz (f)$ is holomorphic in $D$ by Proposition \ref{P:5.4} (because $f\in \mathcal H^{p_1}(\bndry D)$). From this we conclude that $f\in \mathcal H^{p_2}(\bndry D)$ again by Proposition \ref{P:5.4} (because $f\in L^{p_2}(\bndry D)$ and $\Pz (f)$ is holomorphic in $D$).

 \section{The role of the Cauchy integral}\label{S: Hardy-Cauchy} In this section we come to terms with the main issue that arises in the context of several complex variables, namely, the fact that 
there is no canonical, holomorphic Cauchy kernel for $D\subset \Cn$ when $n\geq 2$.
 For this reason we need to 
 impose
  the additional restriction that our domain be strongly pseudo-convex and of class $C^2$,
   so we may apply the results of 
   \cite[Part I]{LS-5}, were the existence (and explicit construction) of a family of holomorphic Cauchy
  kernels is established, along with  $L^p$- and H\"older- regularity properties of the resulting integral operators.

 \subsection{Holomorphic Cauchy integrals}\label{SS:backgr-Cauchy} 
 The proofs of the statements in this section can be found in Part I of \cite{LS-5}. Here we briefly recall the main ideas and ingredients in the proofs.

$\bullet$\quad  Taking $\rho: \Cn\to \R$ to be a strictly plurisubharmonic defining function of $D$,
 one begins by constructing a family of locally holomorphic kernels, denoted $\{C^1_\epsilon (w, z)\}_\epsilon$, by applying the Cauchy-Fantappie' theory, see \cite{LS-3}, to a perturbation of the Levi polynomial of $D$ in which the second derivatives of $\rho$ (which are only continuous functions of $w\in\bar D$) are replaced by a smooth approximation $\tau_\epsilon$. One then achieves global holomorphicity by adding to each such kernel the solution  of an ad-hoc $\deebar$-problem in the $z$-variable (for fixed $w$ in a neighborhood of $\bndry D$), whose data is defined in a (strongly pseudo-convex and smooth) neighborhood $\Omega$ of $\bar D$; we denote such solution $C^2_\epsilon (w, z)$. The outcome of this procedure is a 
 family of globally holomorphic kernels $\{C_\epsilon (w, z)\}_\epsilon$:
    \begin{equation}\label{E:3.3ab}
 C_\epsilon (w, z)= C^1_\epsilon (w, z) +C^2_\epsilon (w, z),\quad 0<\epsilon<\epsilon_0
\end{equation}
which are holomorphic in $z\in D$ whenever $w$ is in $\bndry D$.
 More precisely, each of the $C_\epsilon (w, z)$'s is defined in terms of a denominator $\ge^{-n}(w, z)$, where $\ge (w,z)$ is a holomorphic function of $z\in D$ (and of class $C^1$ in $w\in \Omega$) that satisfies the following
  inequalities uniformly in $0<\epsilon<\epsilon_0$:
 \begin{equation}\label{E:2.4}
  \Re \ge (w, z) \geq c'(-\r (z) + |w-z|^2),\quad \mbox{for}\ z\in\bar D,\ w\in\bndry D\, ,
  \end{equation}
  and
 \begin{equation}\label{E:2.4'}
  \Re \ge (w, z) \geq c'(\r (w)-\r (z) + |w-z|^2)
   \end{equation}
  for $z$ and $w$ in a neighborhood of $\bndry D$.  
 We denote  the resulting integral operators by $\Cine$, that is
  \begin{equation*}
  \Cine f(z) = \int\limits_{w\in\bndry D}\!\!\!\! f(w)\, C_\epsilon (w, z)\, ,\quad z\in D.
  \end{equation*}
   From now on we are only interested in the properties of these operators
  for a fixed (small) value of $\epsilon$.
   Thus we will 
  drop explicit reference to $\epsilon$ and will write $\Cin$ for  $\Cin_\epsilon$, 
 $C(w, z)$ for $C_\epsilon (w, z)$, $\g(w, z)$ for $\ge (w, z)$ and so forth.
 
The key properties of $\Cin$ are summarized in propositions \ref{P:4} and \ref{P:5} below.

   \begin{Prop}\label{P:4}\quad
\begin{itemize}
\item[(1)]\quad Whenever $f$ is integrable, $\Cin (f) (z)$ is holomorphic for $z\in D$.
\item[]
\item[(2)]\quad  If $F$ is continuous in $\bar D$ and holomorphic in $D$ and 
\begin{equation*}
f=F\bigg|_{\bndry D}\, ,
\end{equation*}
\quad then\ \  $\Cin (f) (z) = F(z), \ z\in D.$
\end{itemize}
\end{Prop}

$\bullet$\quad An important feature of the Levi polynomial of $D$ is that it
determines a quasi-distance
 function $\d (w, z)$, defined for $w, z\in \bndry D$, that exhibits border-line integrability: 
\begin{equation}\label{E:2.16}
\int\limits_{w\in \B_r(z)}\!\!\!\!\!\! \d(w, z)^{-2n +\beta} d\sigma (w) \leq c_\beta\, r^\beta; 
\!\!\!\!\!
\int\limits_{w\in \bndry D\setminus \B_r(z)}\!\!\!\!\!\!\!\!\! \d(w, z)^{-2n -\beta} d\sigma (w) \leq c_\beta\, r^{-\beta} 
\end{equation}
for $0<r<1$ and $\beta >0$, where $\B_r(z) =\{w\in\bndry D\ |\ \d(w, z)<r\}$.\\

Then one has the following extension result (proved in Part I of \cite{LS-5}):
 \begin{Prop}\label{P:5}
If $f$ satisfies the H\"older-type condition:
\begin{equation}\label{E:3.4}
 |f(w)-f(z)|\leq c\,\d(w, z)^\alpha,\quad w, z\in \bndry D.
 \end{equation}
 then $\Cin (f)$ extends to a continuous function
on $\bar D$.
\end{Prop}

It follows from this proposition that one may define the Cauchy transform $\Ctre$ as the
restriction of $\Cin$ to the functions on $\bndry D$ that satisfy the H\"older-like condition \eqref{E:3.4}, that is
 \begin{equation*}
\Ctre (f) = \Cin (f)\big|_{\bndry D}\quad \mbox{for}\ f\ \mbox{as in  \eqref{E:3.4}.}
\end{equation*}
Note that with the notation of Section \ref{SS:approx-dom} we have
\begin{equation*}
\Ctre (f) = \dot\Cin (f)\, .
\end{equation*}
The Cauchy transforms have the following regularity properties (proved in Part I of \cite{LS-5}).
\begin{Thm}\label{T:1}
The operator $\Ctre$ initially defined for functions satisfying \eqref{E:3.4} extends to a bounded linear transformation on $L^p(\bndry D, d\sigma)$, for $1<p<\infty$.
\end{Thm}

 \begin{Prop}\label{P:13}
For any $0 < \alpha < 1$, the transform $\Ctre: f\mapsto \Ctr (f)$  preserves the space of H\" older-like functions satisfying condition \eqref{E:3.4}.
\end{Prop}

 \subsection{Stability under small perturbations of the domain} We should note that when $|t|$ is small,  the approximating domains $\Dt$ that were defined in Section \ref{SS:approx-dom}, are strongly pseudo-convex as a consequence of the assumption that $D$ is strongly pseudo-convex.
    It turns out that the Cauchy kernel $\Cker(w, z)$ that was introduced in the previous section
     for the original domain $D$ works as well,
   {\em mutatis mutandis}, for the domains $\{\Dt\}_t$. More precisely, we have:
   
  \begin{Prop}\label{P:aux-Hardy}
   If $|t|$ is sufficiently small, the kernel $\Cker(w,z)$
   given by \eqref{E:3.3ab}
  has a natural extension to $z\in \Dt$ and $\zeta\in\bd\Dt$. 
    The corresponding  Cauchy  integral operator $\Cint$ for $\Dt$, defined as 
    \begin{equation*}
    \Cint (\tilde f) (z) = \int\limits_{\zeta\in\bndry D_t}\!\!\! \tilde f(\zeta)\, C(\zeta, z),\quad z\in D_t,
    \end{equation*}
    satisfies  the following properties:
    \begin{itemize}
\item[(1)]\quad Whenever $\tilde f$ is integrable on $\bndry D_t$, $\Cint (\tilde f) (z)$ is holomorphic

 for $z\in D_t$.
\item[]
\item[(2)]\quad  If $\tilde F$ is continuous in $\bar D_t$ and holomorphic in $D_t$ and 
\begin{equation*}
\tilde f=\tilde F\bigg|_{\bndry D_t}\, ,
\end{equation*}
\quad then\ \  $\Cint (\tilde f) (z) = \tilde F(z), \ z\in D_t.$
\end{itemize}
   \end{Prop}
 Note that this proposition includes the result for the Cauchy integral of $D$ (Proposition \ref{P:4})
   as the case $t=0$.
\begin{proof}To construct  the extension of the kernel
we begin by noting that $\g (w, z)$
still satisfies the inequality analogous to \eqref{E:2.4} for $z\in \bar\Dt$, $\zeta\in\bd\Dt$, when $\rho$ is replaced by $\rt = \rho +t$. Namely,  we have by \eqref{E:2.4'}  that
 $$
 \Re \big(\g (\zeta, z)\big)\geq c(-\rt (z) + |\zeta-z|^2),\quad \mbox{when}\quad z\in \bar\Dt,\ \zeta\in\bd\Dt\, .
 $$
 (Note that using $\rt$ in place of $\rho$ does not change the definition of $\g$.) Hence if we take
 $$
 \CintI (\tilde f)(z) = \int\limits_{\zeta\in\bd\Dt}\!\!\!\! \tilde f(\zeta)\, \CkerI(\zeta, z)\,,\quad z\in\Dt\, ,
 $$
 where
 $$
 \CkerI(\zeta, z)= \frac{G (\zeta, z)\!\wedge\! (\deebar G(\zeta, z))^{n-1}}{\g (\zeta, z)^n}
 $$
 then $\CintI$ is a Cauchy-Fantappi\`e integral for $\Dt$ whose kernel is holomorphic for $z\in\Dt$ close to $\zeta\in\bd\Dt$. Next, as pointed out in e.g., \cite[Lemma 7]{LS-3} and \cite[Proposition 3.2]{LS-2}, there is a smooth, strongly pseudo-convex domain $\O$ that contains $\bar\Dt$ for $|t|$ sufficiently small, with the property that  $H(\zeta, z):= -\deebar_z\CkerI (\zeta, z)$ 
 is smooth when $z\in\O$, and is continuous in $\zeta\in\bd\Dt$. We may thus consider the correction operator $\CinII$ and its kernel $\CkerII (\zeta, z)$ 
as  described in \cite[Part I]{LS-5}
  and references therein, however now for 
 $z\in\O$ and $\zeta\in\bd \Dt$, so that 
 $$
 \deebar_z\CkerII (\zeta, z) = -\CkerI(\zeta, z)\quad \mbox{whenever}\quad z\in \O\, ,\ \zeta\in\bd \Dt\, ,
 $$
 and set 
 $$
 \CintII (\tilde f)(z) =\int\limits_{\zeta\in\bd \Dt}\!\!\!\! \tilde f(\zeta)\, \CkerII (\zeta, z),\quad \mbox{for}\ z\in \Dt \, .
 $$
Then if
   $
  \Cint (\tilde f) = \CintI (\tilde f) + \CintII (\tilde f) \, ,
  $
  we have
  $$
  \Cint (\tilde f) (z) =\int\limits_{\bndry \Dt} \!\!\tilde f(\zeta)\,\Cker (\zeta, z),\quad z\in\Dt
  $$
  where 
  $\Cker (\zeta, z)=\CkerI (\zeta,z) +\CkerII (\zeta, z)$ 
  is in fact the kernel of $\Cin$ (the Cauchy integral for the domain $D$ that was defined in Section 
  \ref{SS:backgr-Cauchy}) and so
  it is independent of $t$.
  It should be clear from the above that  (1) and  (2) holds.
   Namely, $C(\zeta, z)$ is holomorphic in $z\in\Dt$ for any $\zeta\in\bd \Dt$; and 
  if $\tilde F$ is holomorphic in $\Dt$ and continuous in $\bar\Dt$, we have that
 \begin{equation} \label{E:aux-Hardy}
  \tilde F(z) =\int\limits_{w\in\bndry \Dt}\!\!\!\! \tilde f(\zeta)\,\Cker(\zeta, z)\, ,\quad \mbox{for}\quad z\in\Dt\, \quad
  \mbox{and}\ |t|\ \mbox{small},
  \end{equation}
  where
   $$
\tilde f= F\big|_{\bd \Dt}\, .
 $$
 \end{proof}

\subsection{Hardy spaces and the Cauchy Integral} The following proposition gives the first link between the 
 Cauchy integral and $\Hp (D)$.
  \begin{Prop}\label{P:1-Hardy}
 Suppose $f\in L^p(\bd D, d\sigma)$, $1<p<\infty$, and let $F(z) = \Cin f (z)$, $z\in D$. 
 Then, $F\in \Hp (D)$ and 
 $$
 \|F\|_{\Hp (D)}\lesssim \|f\|_{L^p(\bd D, d\sigma)}.
 $$
 Moroever, we have that $\Ctr (f)\in\hp (\bd D, d\sigma)$\, .
 \end{Prop}
 \begin{proof}
 We let $\{g_k\}_k$ be a sequence of smooth (say $C^1$) functions so that $g_k\to f$ in the $L^p(\bndry D, d\sigma)$-norm. Let
 $$
 F_k := \Cin (g_k)\, .
 $$
 Then by Proposition \ref{P:5}, 
  $F_k$ is holomorphic in $D$ and continuous on $\bar D$, and 
 $$
 \dot F_k = F_k\big|_{\bd D} = \Ctr (g_k)\, ,
 $$
 by the definition of the transform $\Ctr$. So $\{F_k\}_k\subset \Hp (D)$, and 
 $$
 \|F_k-F_j\|_{\Hp (D)} \lesssim \|\Ctr(g_k-g_j)\|_{L^p(\bndry D)}
 \lesssim
 \|g_k-g_j\|_{L^p(\bndry D)}\, ,
 $$
 with the first inequality due to \eqref{E:5.2pp}, and the second due to the $L^p$-boundedness of $\Ctr$ (\cite[Theorem 7]{LS-5}).
  This shows that $\{F_k\}_k$ is a Cauchy sequence in $\Hp (D)$ and it follows from \eqref{E:Poisson-bound} and the comments thereafter that $\{F_k\}_k$ 
  has a 
  subsequence (which we keep denoting $\{F_k\}$) that converges uniformly on compact subsets of $D$ and therefore in $\Hp (D)$,  to a limit $F$ which is thus in $\Hp (D)$. Since as we have seen
  $\|F_k\|_{\Hp (D)}\lesssim \|g_k\|_{L^(\bd D)}$, this yields the first assertion.
  
  The fact that $\Ctr(g_k)\in\hp (\bd D, d\sigma)$ (recall that $\Ctr(g_k) =\dot F_k$ with $F_k\in\Hp (D)$), and the continuity of $\Ctr$ in the $L^p(\bd D, d\sigma)$-norm (\cite[Theorem 7]{LS-5}), then shows that 
  $$
  \Ctr (f) = \lim\limits_{k\to\infty}\Ctr (g_k)\in\hp (\bd D, d\sigma)\, .
  $$
  This proves the second assertion, completing the proof of the proposition.
   \end{proof}
The operator $\Cint$ that appeared in Proposition \ref{P:aux-Hardy}
will be used in the proofs of the next two results. Specifically, the situation 
when $t>0$ (for which $\bar \Dt\subset D$) arises in the proof of Proposition \ref{P:2-Hardy} below, whereas
the case $t<0$ (for which $\bar D\subset \Dt$) will occur in the proof of 
Theorem \ref{T:3-Hardy} in the next section.

 \begin{Prop}\label{P:2-Hardy} 
 Suppose $F\in \Hp (D)$ and $f=\dot F$. Then, 
 $$
 F (z) =\Cin (f)(z),\quad z\in D.
 $$
 \end{Prop}
  The assertion above is an elaboration of \cite[Proposition 5]{LS-5}, in which $F$ was taken to be in the subspace of $\Hp (D)$ consisting of the functions that are holomorphic in $D$ and continuous on $\bar D$.
 \begin{proof}
  For $t>0$ we consider the Cauchy integral $\Cint$ for the region $\Dt$ that was defined
 in Section \ref{SS:approx-dom} and then pass to the limit as $t \to 0$.
  
  To this end, first note that by \eqref{E:aux-Hardy},
    for any fixed $z\in D$ we have 
    \begin{equation*}\label{E:5.5}
  F(z) = \int\limits_{\bd \Dt} \!\! F (\zeta)\, \Cker (\zeta, z)\quad \mbox{for}\ \  t>0\quad \mbox{and sufficiently small}
  \end{equation*}
   because $F$ is holomorphic in $\Dt$ and continuous on $\bar\Dt$.
  We may now use the bijection $\Pt$: $\bndry D\to\bndry D_t$ that was described in Section \ref{SS:approx-dom} to 
    express the identity 
  above in the equivalent form
  \begin{equation}\label{E:5.5}
  F(z) = \int\limits_{\bd D} \!\! F (\Pt (w)) J_t(w)\,\Cker (\Pt (w), z)\, 
  \end{equation}
  via a corresponding formulation of the change of variables formula \eqref{E:change-1} (given in the Appendix below).
   Since $J_t\to 1$ uniformly on $\bndry D$, and the coefficients of $\Cker (w, z)$
   are continuous in $w$ in a neighborhood of $\bar D$, 
     then $J_t(w)\Cker (\Pt (w), z)$
  converges to $\Cker (w, z)$ as $t\to 0$, uniformly in $w\in\bd D$;
  moreover 
  the convergence of $\Pt (w)$ to $w$ is non-tangential. 
  Thus \eqref{E:5.2} and the dominated convergence via the maximal function $\N(F)$ shows that the integral
  in \eqref{E:5.5} converges to
  $$
  \int\limits_{\bd D}\!\!\dot F(w)\,\Cker (w, z)\, =\Cin (f) (z).
  $$
  The proposition is therefore proved.
 \end{proof}

\subsection{Density properties of $\hp (\bndry D)$.}
There are two realizations of $\hp (\bndry D)$ that follow from the previous results for $\Cin$ and $\Ctr$.
The first is in fact
 a basic approximation
 of $\hp (\bd D)$. We define $\Ht$  to consist of all
 functions $f$ that arise as restrictions to $\bd D$ of functions $F$ that are holomorphic in some neighborhood of $\bar D$ (which need not be fixed and may, in fact, depend on $F$).
 \begin{Thm}\label{T:3-Hardy}
 For each $p$, $1<p<\infty$, we have that $\Ht$ is dense in $\hp (\bd D)$. 
 \end{Thm}
 In particular, the space of functions that arise as restrictions to $\bd D$ of functions that are holomorphic in $D$ and continuous on $\bar D$, is dense in $\hp (\bndry D)$. \\
   
 \noindent{\em Proof of Theorem \ref{T:3-Hardy}}. 
 We denote by 
 $\Ha$ the space of functions $f$ on $\bd D$ that satisfy the H\"older condition \eqref{E:3.4} and moreover, arise as 
 $$
 f=\dot F\, ,
 $$
 with $F$ holomorphic in $D$ and continuous in $\bar D$. Note that $\Ht \subset \Ha$.
 
 The proof of the proposition is given in two steps: in the first step we show that
  $\Ha$ is dense in $\hp (\bd D)$. Let $f\in \hp (\bd D)$, with $\dot F = f$, where $F\in \Hp (D)$. Note that by Proposition \ref{P:2-Hardy} we have
  \begin{equation*}
  \Cin (f) = F.
  \end{equation*}
  Now let $\{h_n\}$ be a sequence of $C^1$-functions on $\bd D$ (which automatically satisfy \eqref{E:3.4}) so that 
  \begin{equation*}
  h_n\to f
  \end{equation*}
   in the $L^p(\bd D)$-norm, and set $F_n=\Cin (h_n)$.  As in the proof of 
    \cite[Proposition 6]{LS-5},
   we see that $F_n$ is holomorphic in $D$ and continuous in $\bar D$. Let 
  $$
  f_n = \dot F_n = \Ctr (h_n).
    $$
  Then by 
    Proposition \ref{P:13} we have that
  $f_n$ is H\"older-continuous in the sense of \eqref{E:3.4} and we conclude that
  $$
  \{f_n\}_n\subset \Ha\, .
  $$
  
  We claim that 
    $f_n\to f$ in $\hp (\bd D)$. Indeed  by the definitions of $F_n$ and of $F$ we have that
 $$
 F_n-F= \Cin (h_n - f)\, ,
 $$
 and Proposition \ref{P:1-Hardy} thus grants that
 $$
 \| F_n-F\|_{\Hp (D)} \lesssim \|h_n-f\|_{L^p(\bndry D)}\to 0.
 $$
 On the other hand by \eqref{E:5.3} we also have
 $$
 \| F_n-F\|_{\Hp (D)} \approx \| \dot F_n-\dot F\|_{L^p (\bndry D)} = 
 \|f_n-f\|_{L^p(\bndry D)}
 $$
 from which the desired convergence follows.
 
 In the second step of the proof of Theorem \ref{T:3-Hardy} we show that any $f\in \Ha$ can be  approximated uniformly on $\bd D$ by  a family $\{F_t\}_t$ of functions whose boundary values are in $\Ht$. To  construct such a family we use the 
  Cauchy integrals $\{\Cint\}_t$ for the domains $\Dt$ that were defined in Section \ref{SS:approx-dom} for {\em negative} $t$  (note that then $\bar D\subset \Dt$), and apply them to a suitable transposition
of $f$ to  $\bndry D_t$.
 More precisely, given $f\in \Ha$ we define $\widetilde f$ by requiring that
  $\widetilde f(\Pt (w)) = f(w)$ if $t$ is negative and sufficiently small. (Here $\Pt:\bndry D\to \bndry D_t$ is again as in Section \ref{SS:approx-dom}.)
 Now define 
  $$
F_t (z)=  \Cint (\widetilde f)(z),
\quad z\in\Dt,\quad t<0 \quad \mbox{and small,}
 $$
 where we recall that
 $$
  \Cint (\widetilde f)(z) :=\int\limits_{w\in\bd \Dt}\!\!\!\!\! \widetilde f(\zeta)\,\Cker (\zeta, z)
\quad z\in\Dt
$$
is the afore-mentioned Cauchy integral for $\Dt$. Then by part (1) of Proposition \ref{P:aux-Hardy} 
we have that each $F_t$ is holomorphic in $\Dt$ and so its restriction to $\bndry D$ belongs to $\Ht$.

  It will suffice to show that
 $$
  F_t(z) \to f(z)\quad \mbox{uniformly for}\  z\in \bd D,\quad  \mbox{as}\  t\to 0.$$ 
 
 To this end, note that
 \begin{equation}\label{E:int-pre-chg}
   F_t(z) = f(z)+\!\!
  \int\limits_{\zeta\in\bd \Dt}\!\!\!\!\big[\widetilde f(\zeta)-f(z)\big]\,\Cker (\zeta, z)\, ,\quad z\in \bd D\, 
 \end{equation}
 because
 $$
  \int\limits_{\zeta\in\bd \Dt}\!\!\!\! \Cker (\zeta, z)=1\quad \mbox{for}\ z\in \Dt
 $$
 by conclusion (2) of Proposition \ref{P:aux-Hardy}. By an analogous formulation of the change of variables formula \eqref{E:change-1}, the identity \eqref{E:int-pre-chg} can be
 re-written as
 \begin{equation}\label{E:int-post-chg}
  F_t(z)
  = f(z)+\!\!
  \int\limits_{w\in\bd D}\!\!\!\!\big[f(w)-f(z)\big]\,J_t(w)\, \Cker (\Pt(w), z)\, ,\quad z\in \bd D\, .
 \end{equation}
 We point out that a corresponding representation
 for $\Ctr (f)$ was given in \cite[(3.2)]{LS-5}, namely
  \begin{equation}\label{E:int-post-chg-a}
  \Ctr (f)(z) = f(z)+\!\!
  \int\limits_{w\in\bd D}\!\!\!\!\big[f(w)-f(z)\big]\,\Cker (w, z)\, ,\quad z\in \bd D\, .
 \end{equation}
We next remark that whenever $f\in\Ha$, one has
\begin{equation}\label{E:id-Ctr}
\Ctr (f) = f\, .
\end{equation}
 To see this, we write $f$ as $\dot F$, for some $F\in \Hp$ (in fact for $F$ holomorphic in $D$ and continuous on $\bar D$).
 Then it follows by Proposition \ref{P:2-Hardy} that $\Cin f =F$ on $D$; by Proposition \ref{P:5} this identity
 extends to $\bar D$, and \eqref{E:id-Ctr} then follows by the definition of $\Ctr$.
 Combining \eqref{E:id-Ctr} with \eqref{E:int-post-chg-a} we obtain
 \begin{equation*}
   f(z) = f(z)+\!\!
  \int\limits_{w\in\bd D}\!\!\!\!\big[f(w)-f(z)\big]\,\Cker (w, z)\, ,\quad z\in \bd D\, .
 \end{equation*}
 Subtracting the above from \eqref{E:int-post-chg} we find
 \begin{equation*}
  F_t(z) - f(z) = I_t(z) + II_t(z)\,
 \end{equation*}
 where
 \begin{equation*}
 I_t(z)=  \int\limits_{w\in\bd D}\!\!\!\!(f(w)- f(z))\,\big[ \Cker (\Pt(w), z) - \Cker (w, z)\big]\,  ,
  \end{equation*}
  and
   \begin{equation*}
II_t(z)=  \int\limits_{w\in\bd D}\!\!\!\!(f(w)- f(z))\,\big[ J_t(w) -1\big]\, \Cker (\Pt(w), z) ,
 \end{equation*}
 To treat $I_t(z)$, we break the integral on $\bd D$ into two parts: when $\d(w, z)\leq a$, and $\d(w, z)\geq a$. 
 To study the
 integration in $w$ where $\d(w, z)\leq a$, we invoke the following inequality concerning the 
 denominators $\ge (w, z)$ that were described in Section \ref{SS:backgr-Cauchy}, which
 is a consequence
 of  Lemma \ref{L:aux} in the Appendix below,
 \begin{equation}\label{E:aux-for-app}
 |\g (\Pt (w), z)|\gtrsim |\g (w, z)|,\quad \mbox{for} \ \ w, z\in\bd D.
 \end{equation}
 Assuming for now the truth of this inequality, 
 we see 
 by the H\"older-regularity of $f$ 
  that 
  the integrand above is bounded by a multiple of 
 $$
 \frac{1}{\d(w, z)^{2n}}\,\d(w, z)^\alpha \, .
 $$
 Thus, by  \eqref{E:2.16} the integral on the set where $\d(w, z)\leq a$ 
 is majorized by a multiple of 
 $$
 \int\limits_{\d(w, z)\leq a}\!\!\!\!\!\d(w, z)^{-2n+\alpha}\ \lesssim\ a^\alpha\, .
 $$
 
 On the other hand,  by the continuity of
 $\Cker(w, z)$ where $\d(w, z)\geq a$, we see that integration over this set gives a quantity that tends to
 0 uniformly as $t\to 0$. Since $a$ can be chosen arbitrarily small, this shows that the first term $I_t(z)$
 tends to $0$
 as $t\to 0$, uniformly in $z\in\bndry D$.
 The second term $II_t(z)$ can be treated similarly to conclude that $II_t(z)\to 0$ uniformly in $z\in\bndry D$, as well.
 
 Combining all of the above, we conclude that 
   $$
 \sup\limits_{z\in\bd D}|F_t(z) - f(z)|\to 0\quad \mbox{as}\ \ t\to 0^-,
 $$
 and the proposition is established.
  \qed
  
  \medskip
  
The second representation of $\hp (\bd D)$ is as the range of the Cauchy transform
 $\Ctr$.
  \begin{Cor}\label{C:1-Hardy}
  Suppose $h\in L^p(\bd D, d\sigma)$. Then $h\in \hp (\bd D, d\sigma)$ if, and only if, $h=\Ctr (h)$.
  \end{Cor}
  \begin{proof}
  Recall that if $h\in L^p(\bd D, d\sigma)$ then $\Ctr (h)\in \hp (\bd D, d\sigma)$, by Proposition \ref{P:1-Hardy}. Thus, if $h=\Ctr (h)$ we have that $h\in \hp (\bd D, d\sigma)$.
  Conversely if $h\in\hp (\bd D, d\sigma)$, by the density just proved we can approximate it by a sequence
  $\{f_n\}$ with the property that $f_n=\dot F_n$ with $F_n$ holomorphic in $D$ and continuous in $\bar D$.
  Hence, 
  $$
  \Ctr (f_n) = \Cin (\dot F_n)\big|_{\bd D} = f_n
  $$
  where the last equality is due to the identity $\Cin (\dot F_n)(z) = F_n (z)$ for $z$ in $D$ (Proposition 
  \ref{P:2-Hardy}), which extends  to $z$ in $\bar D$ because of the continuity of $F_n$ on $\bar D$. Thus
  $$
  f_n = \Ctr (f_n)\quad \mbox{for each}\ n\, ,
  $$
  and the conclusion $h=\Ctr (h)$ follows by the continuity of $\Ctr$ in the $L^p(\bd D, d\sigma)$-norm.
  \end{proof}
  
     \subsection{Comparing the Cauchy-Szeg\H o projection with the Cauchy integral}
     Proposition \ref{P:1-Hardy} and Corollary \ref{C:1-Hardy} show that $\Ctr$ is a projection:
     $L^p(\bndry D, d\sigma)\to \hp (\bndry D, d\sigma)$. Thus, when $p=2$ we may compare
     $\Ctr$ with the Cauchy-Szeg\H o projection $\S$, which is the {\em orthogonal} projection:
     $L^2(\bndry D, d\sigma)\to \mathcal H^2 (\bndry D, d\sigma)$.

   \begin{Prop}\label{P:6.2.1}
As operators on $L^2(\bndry D, d\sigma)$ we have
\medskip

(a)\quad $\Ctre\S =\S$
\medskip

(b)\quad $\S\Ctre = \Ctre$
\medskip
\end{Prop}

In fact, whenever $f\in L^2(\bndry D)$, then $g=\S f\in \htwo$, by definition of $\S$. If we apply
 Corollary \ref{C:1-Hardy}
   (to $g=\S f$) we see that $\Ctre (\S f) = \Ctre (g) =g = \S f$, proving {\em (a)}. Next, by Proposition \ref{P:1-Hardy}, 
   $\Ctre (f)\in\htwo$, for 
 $f\in L^2(\bndry D)$. Thus $\S\Ctre (f) = \Ctre (f)$, which shows {\em (b)}, thus proving the proposition.
 
 We point out that a corresponding version of Proposition \ref{P:6.2.1} holds for the orthogonal projections: $\S_\omega: L^2(\bndry D, \omega\, d\sigma)\to \mathcal H^2(\bndry D, \omega\, d\sigma)$ for the densities $\omega$ discussed in the introduction (so in particular for the Leray-Levi measure).
 \section{Further results}\label{S:further}

  To conclude, we give a further characterization of $\hp (\bndry D, d\sigma)$ as the range of the Cauchy-Szeg\H o projection $\S$. This uses the $L^p(\bndry D)$-regularity
   of the Cauchy-Szeg\H o projection $\S$, proven in \cite[Theorem 16]{LS-5}, and the approximation theorem just proved (Theorem \ref{T:3-Hardy}).
 
  \begin{Prop}\label{L:Har-Sz}
  Suppose $f\in L^p(\bd D)$, $1<p<\infty$. Then $\S (f)\in \hp (\bd D)$. Furthermore,  we have that $f\in \hp (\bd D)$ if, and only if, $f=\S (f)$.   \end{Prop}
  \begin{proof} To prove the first conclusion, we  consider first the case when $p\geq2$. 
  Then $\S (f)\in \htwo\cap L^p (\bndry D)$ (because $L^p\subseteq L^2$ and $\S: L^2\to \htwo$) and thus
  $\S (f) \in \hp (\bndry D)$ by Corollary \ref{C:5.5}. In the case when $p<2$, 
  we take $(f_n)_n\subset C^1(\bndry D)$ with $\|f_n-f\|_p\to0$. Then $\S (f_n)\in \mathcal H^q(\bndry D)$
  for any $q>2$ (by the case just proved, since $f_n\in C^1(\bndry D)\subset L^q(\bndry D)$ and $q>2$). But $\mathcal H^q(\bndry D)\subset
  \hp (\bndry D)$ (because $q>2>p$) and so $\{\S (f_n)\}_n\subset \hp (\bndry D)$.
 Furthermore $\|\S(f_n-f)\|_p\lesssim \|f_n-f\|_p$ by the $L^p(\bndry D)$-regularity of $\S$.
    We conclude that $\S (f)\in \hp (\bndry D)$ by the completeness of $\hp (\bndry D)$.
 This proves the first conclusion of the proposition. To prove the second conclusion, suppose first that
  $f\in \hp (\bd D)$; then by Theorem \ref{T:3-Hardy} there is a sequence
  $\{f_n\}\subset \Ht$ such that $f_n\to f$ in $L^p(\bndry D, d\sigma)$. But $\Ht\subset \mathcal H^2(\bndry D)$ and thus $\S f_n = f_n$ (by the definition of $\S$). On the other hand, $\S f_n \to \S f$ in $L^p(\bndry D)$ by the regularity of $\S$ in 
  $L^p(\bndry D)$.
   Thus $\S f = f$.
  Conversely, if $f\in L^p(\bndry D)$ and $f=\S (f)$, then $f\in \hp (\bd D)$ by the first
  conclusion of the proposition.
   \end{proof}
   \section{Appendix}\label{S:app}
   \bigskip

Consider the $C^1$-smooth vector field $\nu$ defined in a neighborhood of $\bndry D$ by 
 $$
 \nu (h) (x) =\frac{1}{|\nabla\rho (x)|^2}\sum\limits_{j=1}^{2n}\frac{\dee\rho}{\dee x_j}(x)\,\frac{\dee h}{\dee x_j}(x)\, \quad x\in U,\quad h\in C^1(U).
 $$
 Note that $\nu(\rho) (x) \equiv 1$. Let $\Pt (x) = \exp (t\,\nu) (x)$ be the exponential map associated with $\nu$,
defined on a small neighborhood $U'$ of $\bndry D$, for small $t$. One knows that
 $$
 x\mapsto\Pt (x),\quad x\in U'
 $$
 is a $C^1$-smooth mapping when $t$ is small. We recall that $\Pt (\cdot) (x)$ arises as the solution of the time-independent 
   differential equation
 \begin{equation}\label{E:1A}
 \frac{d}{dt}\left(\Pt (\cdot)(x)\right) = a\left(\Pt (\cdot) (x)\right)\quad \mbox{with}\quad a(u)=
 \frac{\nabla\rho (u)}{|\nabla\rho (u)|^2},
 \end{equation}
 with initial condition 
 $$
 \Pt  (x)\big|_{t=0} = x\, .
 $$
 As a result,
 $$
 \frac{d}{dt}\left(\rho(\Pt (x))\right) \equiv 1\, ,
 $$
 and hence $\rho(\Pt (x)) = \rho (x) + t = \rt (x)$.

  The existence, uniqueness and $C^1$-regularity of the solution to equations like \eqref{E:1A} guarantee the following properties:
  
 \begin{itemize}
 \item \quad $\Phi_{t_1}\circ\Phi_{t_2}= \Phi_{t_1+t_2}\quad \mbox{if the}\quad t_j's \quad \mbox{are small;}$
 \item \quad $\Phi_0 =$ Identity.
  \end{itemize}
Note that these  properties imply that $\Pt$ is a $C^1$-bijection:  $\bndry D \to \bndry \Dt$ with inverse $\Phi_{-t}$. Moreover $J_t (x)$ (the Jacobian determinant of $\Pt$)  has the following properties:
\begin{itemize}
\item \quad $0<c_1\leq J_t(x)\leq c_2<\infty$ uniformly in $t$ and $x\in\bndry D$;
\item\quad $|J_t(x) -1|\to 0$ as $t\to 0$, uniformly in $x\in \bndry D$;
\item \quad The following change of variable formula holds:
\begin{equation}\label{E:change-1}
\int\limits_{\bndry D}\!\!\tilde f (\Pt (w))\, J_t(w)\, d\sigma (w)  =\!\!
\int\limits_{\bndry D_t}\!\!\tilde f (\zeta)\, d\sigma_t (\zeta)
\end{equation}

whenever $\tilde f$ is integrable on $\bndry D_t$.
\end{itemize}
\subsection{ Proof of inequality \eqref{E:aux-for-app}} The proof of this inequality is a consequence of the following lemma.
  \begin{Lem}\label{L:aux} For $t<0$, $t$ small, we have
 $$
 |\g (\Pt (w), z)|\approx |\g (w, z)| + |t|\, , \ \  \mbox{for}\quad w, z\in\bd D\, .
 $$
 \end{Lem}
 \begin{proof}
 Without loss of generality we will assume that $w$ is close to $z$; then we may 
  write $\g (w, z) = \langle\dee\rho (w), w-z\rangle + Q_w(w-z)$, with 
 $Q_w (u)$ a quadratic form in $u$.  
 
 Letting $\nu_w$ denote the inner unit normal vector at $w\in\bndry D$, we claim that
 \begin{equation}\label{E:nuw}
 \g (w -\delta\nu_w, z) = \g (w, z+\delta\nu_z) + O(\delta|w-z|+\delta^2)
 \end{equation}
 
 when $\delta>0$ is sufficiently small. To prove this claim, we begin by noting 
 that
 $$
 \langle\dee\rho (w-\delta\nu_w), w-\delta\nu_w-z\rangle =
 \langle\dee\rho (w), w-\delta\nu_w-z\rangle + O(\delta |w-z| +\delta^2) 
 $$
 $$
 = \langle\dee\rho (w), w-z-\delta\nu_z\rangle + O(\delta |w-z| +\delta^2)\, . $$
 One similarly has
 $$
 Q_{w-\delta\nu_w}(w-\delta\nu_w-z)= Q_w(w-z+\delta \nu_z)+O(\delta|w-z|+\delta^2)\, ,
 $$
 which combined with the above proves the claim.
 
 Next, we observe that the first-order Taylor expansion at $t=0$ of the bijection $\Pt(w)$ (as a function of $t$, for fixed $w\in \bndry D$) is
 $$
 \Pt(w) = w+t\N_w+o(|t|)\quad \mbox{as}\ \ t\to 0
 $$
 with $o(|t|)$ uniform as $w$ ranges over $\bd D$, where
  $$
 \N_w=\frac{d\Pt (w)}{dt}\big|_{t=0}=\, \frac{\nu_w}{|\nabla\rho (w)|}\, .
 $$
 As a result we obtain
 $$
 \g (\Pt(w), z) = \g (w-\delta\nu_w, z) + o(\delta)\, \ \ \mbox{with}\  \ \delta =-\frac{t}{|\nabla\rho (w)|}>0\quad \mbox{and small}.
 $$
 By \eqref{E:nuw}
 we have
 $$
 |\g (w-\delta\nu_w, z)|
  =
 |\g (w, z+\delta\nu_z)|+ O(\delta|w-z|+\delta^2) +o(\delta).
 $$
 And by \cite[Corollary 2]{LS-5}
 $$
 |\g (w, z+\delta\nu_z)|+ O(\delta|w-z|+\delta^2) + o(\delta) \approx
 $$
 $$
 \approx
 |\g (w, z)| + \delta + o(\delta) + O(\delta|w-z|+\delta^2)  \approx |\g (w, z)| +|t|,
 $$
 since $\delta \approx |t|$ for $t$, and hence $\delta$, sufficiently small. 
\end{proof}

\end{document}